\documentclass[12pt, twoside, leqno]{article}
\usepackage{amsmath,amsthm}
\usepackage{amssymb,latexsym}
\usepackage{enumerate}

\usepackage{amsfonts}

\usepackage[cp1250]{inputenc}

 \numberwithin{equation}{section}

\newcommand{\su}{{\rm  supp\,\, }}

\newcommand{\ve}{\mbox{\scriptsize \rm{V}}}
\newcommand{\es}{\mbox{\scriptsize \rm{S}}}
\newcommand{\ce}{\mbox{\scriptsize \rm{C}}}

\newcommand{\p}{\mbox{\scriptsize $\rm{\Pi}$}}
\newcommand{\pj}{{\mathrm{\Pi}}}

\newcommand{\ddr}{${\cal D}'^{ \{ M_{p } \} }$}

\newcommand{\md}{{\cal D}^{\{M_{p }\}}}

\newcommand{\mdd}{{\cal D}'^{\{M_{p }\}}}

\newcommand{\me}{{\cal E}^{\{M_{p}\}}}

\newtheorem{lemma}{Lemma}
\newtheorem{remark}{Remark}
\newtheorem{proposition}{Proposition}
\newtheorem{theorem}{Theorem}

\renewcommand{\(}{\left(}
\renewcommand{\)}{\right)}

\begin{document}

 \title {Convolution of Roumieu ultradistributions in sequential approach}

 \author{Svetlana Mincheva-Kami\'nska}


\maketitle

 {\hspace{-.6cm}\textsl{Institute of Mathematics,
 Faculty of Mathematics and Natural Sciences, \\
\hspace{.3cm} University of Rzesz\'ow, \\
\hspace{.3cm} Prof. Pigonia 1, 35-310 Rzesz\'ow, Poland} \\
 \hspace{.3cm} \textsl{email}: \emph{minczewa@ur.edu.pl}

\vspace{1.2cm}

 \begin{abstract}
  We consider several general sequential conditions for convolvability
 of two Roumieu ultradistributions on $\mathbb{R}^d$  in the space $\mathcal{D}'^{\{M_p\}}$
 and prove that they are equivalent to the convolvability of these ultradistributions
 in the sense of Pilipovi\'{c} and Prangoski. The discussed conditions, based on two classes
 ${\mathbb U}^{\{M_p\}}$ and $\overline{\mathbb U}^{\{M_p\}}$
 of approximate units and corresponding sequential conditions of integrability of Roumieu ultradistributions,
 are analogous to the known convolvability conditions in the space $\mathcal{D}'$ of distributions
 and in the space $\mathcal{D}'^{(M_p)}$ of ultradistributions of Beurling type.

Moreover, the following  property of the convolution and ultradifferential operator $P(D)$
 of class $\{M_p\}$ is proved: if $S, T \in \mdd(\mathbb{R}^d)$ are convolvable, then
 \[
 P(D)(S\ast T) = (P(D)S)\ast T = S\ast(P(D)T).
 \]
  \\



 \end{abstract}

 \section{Introduction}\label{sec1}

Deep investigations of the convolution of two ultradistributions of Roumieu type
(that we call shorter Roumieu ultradistributions)
in the non-quasianalytic case were carried out via $\varepsilon$-tensor product
by Pilipovi\'{c} and Prangoski in \cite{PP} and,
with important
improvements, by Dimovski, Pilipovi\'{c}, Prangoski and Vindas in \cite{DPPV}.
The authors gave there general functional definitions and proved fundamental results on convolvability
and the convolution of Roumieu ultradistributions in a way analogous to the known general approaches
of Chevalley and Schwartz in case of distributions.
For other aspects of the theory
see e. g.  \cite{CKP,DPV,DPrV,PP15,VV}. See also the recent article \cite{PPV} for results
concerning the quasianalytic case.

The aim of this paper is to discuss sequential conditions playing a similar role in the study of the convolution
of Roumieu ultradistributions to those used in the sequential theories of the convolution of distributions
(see \cite{VLA1}, \cite{DiVo}, \cite{Kam} and \cite{M-K}) and ultradistributions of Beurling type (see \cite{KKP}, \cite{KPP} and \cite{CKP}).
The conditions are based on two types of $\mathfrak{R}$-approximate units (Definition\,2
and Definition\,3),
being the counterparts of the approximate units in the sense of Dierolf and Voigt (see \cite{DiVo}).
The respective classes ${\mathbb U}^{\{M_p\}}$ and $\overline{\mathbb U}^{\{M_p\}}$
 of $\mathfrak{R}$-approximate units are used in sequential characterization of integrable Roumieu ultradistributions (\cite{SMK}),
 analogous to those proved by Pilipovi\'{c} in \cite{Pil91} in case of integrable ultradistributions of Beurling type.
As a consequence, we give
several sequential definitions of the convolution of Roumieu ultradistributions (Definition\,6).
We prove in Theorem\,\ref{thm2}, that all our sequential definitions of the convolution of Roumieu ultradistributions
 are equivalent to those given in \cite{PP} and \cite{DPPV}.

An important application of the notion of $\mathfrak{R}$-approximate units is presented in the proof of Theorem\,\ref{thm3},
describing a non-trivial property of the convolution of Roumieu ultradistributions and ultradifferential
operators of the class $\{M_p\}$.

 It is worth to recall that Pilipovi\'{c} in \cite{Pil91} used a different but analogous
 class of approximate units to prove the same property in case of the convolution of ultradistributions of Beurling type.
 Our proof of Theorem\,\ref{thm3} is based on similar ideas but discussion concerning the class
 $\mathfrak{R}$ plays an essential role in our case.



 \section{Preliminaries}

    We consider complex-valued $\mathcal{C}^{\infty}$-functions and Roumieu ultradistributions defined on $\mathbb{R}^d$
    (or on an open subset of $\mathbb{R}^d$) using the standard multi-dimensional notation
    in $\mathbb{R}^d$.

  To mark the dimension of $\mathbb{R}^{d}$, which is essential in some situations, we denote
 the considered spaces of test functions and the corresponding spaces
 of Roumieu ultradistributions simply by adding the index $d$
 at the end of the respective symbol. Moreover, if necessary, the constant function $1$ on
 $\mathbb{R}^{d}$ will be denoted by $1_d$ and the value of $T\in{\cal{D}'}^{\{M_p\}}_d$
 on $\varphi\in{\cal{D}}^{\{M_p\}}_d$ by $\langle T, \varphi\rangle_d$.

\smallskip
The spaces of test functions and Roumieu ultradistributions are defined by a given
sequence
$(M_{p})_{p\in\mathbb{N}_0}$ of positive numbers. Usually some of the following conditions
are imposed on the sequence $(M_{p})$:

\medskip
 \noindent
 (M.1)\ (logarithmic convexity)

 \quad \qquad $M^{2}_{p} \leq M_{p-1} M_{p+1},\qquad p \in \mathbb{N};$

 \medskip
 \noindent
 (M.2)\ (stability under ultradifferential operator)

 \quad \qquad $M_{p}\leq AH^{p}M_qM_{p-q},\qquad p,\,q\in \mathbb{N}_0,\; q\leq p;$

 \medskip
 \noindent
 (M.2') \ (stability under differential operator)

 \quad  \ \quad $M_{p}\leq AH^{p}M_{p-1},\qquad p\in \mathbb{N};$

 \medskip
 \noindent
 (M.3)\ (strong non-quasi-analyticity)

 \quad \qquad $\sum^{\infty}_{p=q+1} M_{p-1}M_{p}^{-1} \leq Aq M_{q}M_{q+1}^{-1},\quad q \in \mathbb{N};$

 \medskip
 \noindent
 (M.3') \ (non-quasi-analyticity)

 \quad  \ \qquad $\sum^{\infty}_{p=1} M_{p-1}M_{p}^{-1} < \infty,$

 \medskip
 \noindent
 for certain constants $A>0$ and $H>0$. We can and will assume that $H\geq1$.

 Clearly, conditions (M.2') and (M.3') are particular cases of conditions (M.2) and (M.3),
 respectively.

 For simplicity, we will assume in the whole paper that the sequence $(M_{p})$ satisfies the three
 conditions (M.1), (M.2) and (M.3), not discussing which of them can be weakened or omitted
 in the formulations of presented theorems.

It follows, by induction, from (M.1) that
$M_p\cdot M_q \leq M_0 M_{p+q}$ for $p, q\in \mathbb{N}_0$ (see \cite{SMK}).  Under the assumption that $M_0 = 1$, which we adopt hereinafter
for simplicity,
the last inequality admits the form:
 \begin{equation}\label{Mpq}
   M_p\cdot M_q \leq M_{p+q}, \qquad p, q\in \mathbb{N}_0.
     \end{equation}

 It will be convenient to extend the sequence $(M_{p})_{p\in\mathbb{N}_0}$ to
 ($M_k)_{k\in\mathbb{N}_0^d}$ by means of the formula:
 \[
  M_k := M_{k_1+\ldots+k_d},\qquad k = (k_1,\ldots,k_d)\in \mathbb{N}_0^d.
 \]
Due to the extended notation we immediately get the extended version of inequality (\ref{Mpq}):
 \begin{equation}\label{Mkjd}
   M_j\cdot M_k \leq M_{j+k}, \qquad j, k\in \mathbb{N}^d_0.
     \end{equation}

 The {associated function} of the sequence  $(M_p)$ is given by
\[
M(\rho)= \sup_{p\in\mathbb{N}_0} \log_+\frac{\rho^p}{M_p},  \quad \rho>0.
\]
\smallskip

For an arbitrary $k = (k_1,\ldots,k_d)\in\mathbb{N}_0^d$
denote by $D^k$ the differential operator of the form

\[\displaystyle D^k=D_1^{k_1}\cdots D_d^{k_d}:=
\left(\frac{1}{i}\frac{\partial}{\partial x_1}\right)^{k_1}\cdots\left(\frac{1}{i}\frac{\partial}{\partial x_d}\right)^{k_d}.
\]
\smallskip

An essential role in our considerations will play Komatsu's lemma proved in \cite{Kom3}
(see Lemma\,3.4 and Proposition\,3.5) in which numerical sequences mo\-no\-tono\-usly increasing
to infinity are involved. The class of such sequences
$(r_{\!p})_{p\in\mathbb{N}_0}$ (with $r_0=1$)
has been denoted by $\mathfrak{R}$ in \cite{PP} and  \cite{DPPV} and we preserve this notation in our paper.

 For every $(r_{\!p})\in\mathfrak{R}$ we call $(R_p)$ the \textit{product sequence}
 corresponding to  $(r_{\!p})$ if its elements are of the form
 $R_p:=\prod_{i=0}^p r_i$ for $p\in\mathbb{N}_0$ (i.e. $R_0=1$).

 Let us recall Komatsu's lemma in the following equivalent form which emphasizes the symmetry of two assertions:

\begin{lemma}\label{Kom}
 Let $(a_k)_{k\in\mathbb{N}_0}$ be a sequence of nonnegative numbers.

\medskip
 $(I)$\quad\; The following two conditions are equivalent:
   \begin{eqnarray*}
  \vspace{-.1cm}
(A_1) \qquad \exists_{h> 0} \;\;\;\;  &&  \sup_{k\in\mathbb{N}_0}\frac{a_k}{h^k} <\infty;
 \hspace{4cm} \\
\vspace{-.3cm}
 (B_1) \qquad \forall_{(r_k)\in \mathfrak{R}}  &&  \sup_{k\in\mathbb{N}_0}\frac{a_k}{R_k} <\infty;
 \hspace{4cm}
  \end{eqnarray*}

 \vspace{-.1cm}
 $(I\!I)$\quad\;  the following two conditions are equivalent:
   \begin{eqnarray*}
 (A_2) \qquad \forall_{h> 0} \;\;\;\; &&  \sup_{k\in\mathbb{N}_0}(h^k a_k) <\infty;  \hspace{3.3cm} \\
\vspace{-.3cm}
 (B_2) \qquad \exists_{(r_k)\in \mathfrak{R}}  &&
   \sup_{k\in\mathbb{N}_0}(R_k a_k) <\infty,  \hspace{3.3cm}  \label{Ch_p}
  \end{eqnarray*}
 where $(R_k)$ is the product sequence corresponding to the sequence
    $(r_k)\in \mathfrak{R}$.
 \end{lemma}

\begin{remark} \label{remK}
The above lemma can be easily extended to the  $d$-dimensional version
 concerning sequences $(a_k)_{k\in\mathbb{N}_0^d}$ of nonnegative numbers.
\end{remark}

It is worth noticing that Lemma\,\ref{Kom} delivers two simple characterizations (dual to each other):
$1^{\circ}$  of \textit{slowly increasing sequences} (i.e. satisfying $(A_1)$), \
$2^{\circ}$  of \textit{rapidly decreasing sequences} (i.e. satisfying $(A_2)$) of nonnegative numbers.
They are expressed through respective properties of sequences, described by product sequences corresponding
to sequences of the class $\mathfrak{R}$.

In what follows we will also apply the following simple lemma (see \cite{SMK}):

\begin{lemma}\label{Rk-ineq}
For every $(r_p)\in \mathfrak{R}$,
the following inequality holds:
\begin{equation}\label{Rpq}
R_{|k|}\cdot R_{|l|} \leq R_{|k+l|},\qquad\quad k,l\in \mathbb{N}_0^d,
\end{equation}
where $(R_p)$ is the product sequence corresponding to $(r_p)$.
\end{lemma}

It will be convenient to use for $\lambda>0$
and $(r_{\!p})\in\mathfrak{R}$ the
following notation:
 \begin{equation}\label{lamrp}
 \lambda(r_{\!p}) = (\overline{r}_{\!p}), \; \mbox{where} \;\; \overline{r}_{0}=1  \; \; \mbox{and} \;\;
 \overline{r}_{\!p} = \lambda r_{\!p} \;  \; \mbox{for} \; \; p\in\mathbb{N}.
 \end{equation}
 \noindent

 Clearly, if $(r_{\!p})\in\mathfrak{R}$, then
 $\lambda(r_{\!p})\in\mathfrak{R}$ for $\lambda \geq 1$. Moreover, $\lambda(r_{\!p})\in\mathfrak{R}$
 for $0 < \lambda < 1$ in case  $(r_{\!p})\in\mathfrak{R}$ and $r_1 > \lambda^{-1}$.




\section{Ultradifferentiable Functions}


   For a given complex-valued function $\varphi$ on $\mathbb{R}^{d}$ and a compact set $K$ in $\mathbb{R}^{d}$ denote
   \[
   \|\varphi\|_{\infty} := \sup_{x\in\mathbb{R}^{d}} |\varphi (x)|;\qquad \|\varphi\|_K := \sup_{x\in K} |\varphi (x)|.
   \]
  \smallskip
 For a given sequence
 $(M_p)$, a regular compact set $K$
 in $\mathbb{R}^{d}$
 and $h > 0$ the symbol
 ${\cal{E}}^{\{M_{p}\}}_{K,h,d} $
 will mean the locally convex space (l.c.s.)
 of all $\mathcal{C}^{\infty}$-functions
 $\varphi$ on $\mathbb{R}^{d}$ such that
 \begin{equation}\label{1qn}
  q_{K, h}(\varphi) := \sup_{k \in \mathbb{N}^{d}_0}
 \frac{
 \| D^k\varphi \|_K} {h^{|k|}M_{k}} < \infty,
\end{equation}
 with the topology defined by
 the semi-norm $q_{K, h}$ given above,
 while the symbol ${\cal{D}}^{\{M_{p}\}}_{K,h,d}$
 will mean the Banach space of all $\cal{C}^{\infty}$-functions $\varphi$ satisfying (\ref{1qn})
  and having  supports contained in  $K$,
  with the topology of the norm $q_{K, h}$ in (\ref{1qn}).

 For a fixed sequence
 $(M_p)$, we consider the following
            locally convex spaces
 of ultradifferentiable functions on $\mathbb{R}^{d}$:

 \begin{equation}\label{D-K}
 {\cal D}_{K,d}^{\{M_p\}}
 \!:=\;
  \lim_{\begin{subarray}{c}
    \displaystyle\longrightarrow \\
    {\scriptsize h\rightarrow \infty}
    \end{subarray}} \,
 {\cal D}_{K,h,d}^{\{M_p\}}\, ;  \qquad
  {\cal D}_d^{\{M_p\}}
 \,\!:=\!
   \lim_{\begin{subarray}{c}
   \displaystyle\longrightarrow \\
   {\scriptsize K\subset\subset \mathbb{R}^d}
   \end{subarray}} \,
   {\cal D}_{K,d}^{\{M_p\}};
 \end{equation}
 \begin{equation}\label{E}
 {\cal E}_d^{\{M_p\}} :=
 \lim_{\begin{subarray}{c}
   \displaystyle\longleftarrow \\
   {\scriptsize K\subset\subset \mathbb{R}^d}
   \end{subarray}} \,
  \lim_{\begin{subarray}{c}
    \displaystyle\longrightarrow \\
    {\scriptsize h\rightarrow \infty}
    \end{subarray}} \;
  {\cal E}_{K,h,d}^{\{M_p\}}\, ,
  \end{equation}
\noindent
where the symbol $K\subset\subset \mathbb{R}^d$
means that compact sets $K$
grow up to $\mathbb{R}^d.$

 Moreover, for a given $(M_p)$, we define
\[
 {\cal D}_{L^{\infty},d}^{\{M_p\}} \,:=\,
 \lim_{\begin{subarray}{c}
    \displaystyle\longrightarrow \\
    {\scriptsize h\rightarrow \infty}
    \end{subarray}} \,
 {\cal D}_{L^{\infty},h,d}^{\{M_p\}}
 \vspace{-.2cm}
 \]
 where ${\cal{D}}^{\{M_p\}}_{L^{\infty},h,d}$
 is the Banach space of
 all $\mathcal{C}^{\infty}$-functions $\varphi$
 on $\mathbb{R}^d$ such that
   \begin{equation}\label{DL-inf}
   \|\varphi\|_{\infty,h} :=\, \sup
   \left\{ (h^{k}M_k)^{-1}\|D^k\varphi\|_{\infty}\!:\ k \in \mathbb{N}^{d}_{0}\right\} < \infty,
  \end{equation}
 with the norm $\|\cdot\|_{\infty,h}$ defined above.

\smallskip
 For a given regular compact set
  $K\subset\mathbb{R}^d$ and given sequences $(M_p)$ and $(r_p)\in\mathfrak{R}$ we denote by
 ${\cal D}^{\{M_p\}}_{K,(r_p),d}$
 the Banach space of all $\mathcal{C}^{\infty}$-functions $\varphi$ on $\mathbb{R}^d$
 having supports contained in $K$ such that
    \begin{equation}\label{DK-ri}
    \|\varphi\|_{K,(r_p)} := \sup_{k \in \mathbb{N}^{d}_0}
        \frac{\| D^k\varphi\|_K}
        {R_{|k|}M_{k}} < \infty
\end{equation}
 with the norm $\|\cdot\|_{K,(r_p)}$ defined above.

The following result is essentially due to Komatsu \cite{Kom3} (see also \cite{SMK} for the proof), since it is a consequence of
his
Lemma \ref{Kom} recalled above.

\begin{proposition} \label{p1}
    We have the equality
\[
 \displaystyle {\cal D}_{K,d}^{\{M_p\}} =
 \lim_{\begin{subarray}{c}
   \displaystyle\longleftarrow \\
   {\scriptsize (r_p)\in\mathfrak{R}}
   \end{subarray}} \,
 \, {\cal D}_{K,(r_p),d}^{\{M_p\}}\, ,
 \]
 where the space $\displaystyle {\cal D}_{K,d}^{\{M_p\}}$ is  defined in (\ref{D-K}).
\end{proposition}


 For given $(M_p)$ and
$(r_p)\in\mathfrak{R}$ we denote by
${\cal D}_{L^{\infty},(r_p),d}^{\{M_p\}}$
 the Banach space of all $\mathcal{C}^{\infty}$-functions
 $\varphi$ on $\mathbb{R}^d$
  such that
     \begin{equation}\label{D-ri}
        \|\varphi\|_{(r_p)} := \sup_{k \in \mathbb{N}^{d}_0}
       \frac{\|D^k\varphi\|_{\infty}}{{R_{|k|}M_{k}}} < \infty,
     \end{equation}
 with the norm $\|\cdot\|_{(r_p)}$ defined in (\ref{D-ri}).

\smallskip
 For a given sequence $(M_p)$
the following projective description od the space ${\cal D}_{L^{\infty},d}^{\{M_p\}}$ follows from the results proved
in \cite{DPPV,PPV}:
 \[
  \displaystyle \displaystyle
  {\cal D}_{L^{\infty},d}^{\{M_p\}}
  \,=\! \lim_{\begin{subarray}{c}
   \displaystyle\longleftarrow \\
   {\scriptsize (r_p)\in\mathfrak{R}}
   \end{subarray}} \,
 {\cal D}_{L^{\infty},(r_p),d}^{\{M_p\}},
  \]
where the equality holds in the sense of l.c.s.

     We denote by $\dot{\cal{B}}^{\{M_p\}}_d$
 the completion of ${\cal{D}}^{\{M_p\}}_d$ in
	${\cal{D}}^{\{M_p\}}_{L^{\infty},d}$.

In \cite{SMK}, the following assertion concerning the product of functions in ${\cal{D}}^{\{M_p\}}_{L^{\infty},d}$ is proved:

\begin{proposition} \label{p2}
If $\varphi_1, \varphi_2\in {\cal{D}}^{\{M_p\}}_{L^{\infty},d}$,
then
$\varphi_1\cdot\varphi_2\in{\cal{D}}^{\{M_p\}}_{L^{\infty},d}$.
Moreover, for every $(r_{\!p})\in\mathfrak{R}$ such that $r_1>2$
the inequality holds:
   \begin{equation}\label{no_fiphi}
      \|\varphi_1\cdot\varphi_2\|_{(r_p)} \leq    \|\varphi_1\|_{{(r_{\!p})}\!/{2}}\cdot \|\varphi_2\|_{(r_{\!p})\!/{2}},
     \end{equation}
 where $({r_{\!p}})/{2}$ is meant in the sense of (\ref{lamrp}).
\end{proposition}

 \begin{remark} \label{fiphiD}
The assertion of Proposition\,\ref{p2} is true
for functions
 $\varphi_1, \varphi_2$ from the space ${\cal{D}}^{\{M_p\}}_d$
 and
 semi-norms
 (\ref{DK-ri}).
    \end{remark}

  \begin{remark} \label{rem0rn}
 We may assume, if necessary, that the considered sequence $(r_{\!p})\in\mathfrak{R}$ satisfies for any given constant $c>0$
 the inequality $r_{\!p}>c$ for all $p\in\mathbb{N}$. In fact, we have  $\|\varphi\|_{(r_{\!p})}<\infty$ if and only if  $\|\varphi\|_{(\widetilde{r}_{\!p})}<\infty$ for
 all $\varphi\in{\cal{D}}^{\{M_p\}}_d$ with the sequence $(\widetilde{r}_{\!p})\in\mathfrak{R}$  defined by
 $\widetilde{r}_0=1$ and $\widetilde{r}_{\!p}:= r_{\!p+p_0}$ for all $p\in\mathbb{N}$, where  $p_0\in\mathbb{N}$ is an index  such that $r_{\!p}>c$  for $p>p_0$.
  \end{remark}




\section{Roumieu Ultradistributions}


 \textbf{Definition 1}. \label{ultraRo}
  We denote  the strong dual of the space\ ${\cal{D}}^{\{M_{p}\}}_d$
  by \ ${\cal{D'}}^{\{M_{p}\}}_d$
  and call it the {\it space
  of Roumieu ultradistributions}.

 The  strong dual of the space\ $\dot{{\cal{B}}}^{\{M_p\}}_d$,
 denoted by\ ${\cal{D'}}^{\{M_p\}}_{L^{1},d}$
 is called
 the space of {\it Roumieu integrable ultradistributions}.

\begin{remark} \label{rem1}
The space\ ${\cal{D}}^{\{M_p\}}$\ is
 dense in\ $\dot{{\cal{B}}}^{\{M_p\}}$\ and the respective inclusion  mapping is continuous.
 Consequently, the space\ ${\cal{D'}}^{\{M_p\}}_{L^{1}}$\ of
 Roumieu integrable ultradistributions
        is a subspace of the space\ \ddr\  of
Roumieu ultradistributions.
        \end{remark}

In the sequel, we assume that the sequence $(M_p)$ satisfies conditions (M.1), (M.2') and (M.3'), so
we can use the projective description of the considered locally convex spaces of test functions.

 \textbf{Definition 2}. \label{RAU}
 By an {\it $\mathfrak{R}$-approximate unit} we mean a sequence
$(\p_{n})$ of ultradifferentiable functions
$\p_{n}\in\md_d$
converging to $1$  in $\me_d$  such that the following property
holds for every sequence $(r_{\!p})\in\mathfrak{R}$:
 \begin{equation}\label{RRap}
 \sup_{n \in \mathbb{N}} \|\p_n\|_{(r_{\!p})} =
 \sup_{n \in \mathbb{N}} \sup_{k \in \mathbb{N}_0^{d}}
	 (R_{|k|}M_{k})^{-1} \|D^k\p_n\|_{\infty}  < \infty,
 \end{equation}
where $(R_p)$ is the product sequence corresponding to $(r_p)$.

 \textbf{Definition 3}. \label{RsAU}
 By a {\it special $\mathfrak{R}$-approximate unit} we mean an $\mathfrak{R}$-approximate unit
 $(\p_{n})$ such that for every compact set $K\subset \mathbb{R}^{d}$ there
 exists an index $n_0 \in\mathbb{N}$ such that $\p_{n}(x)=1$ for all $n\geq n_0$ and $x\in K$.

We denote the class of all $\mathfrak{R}$-approximate units on $\mathbb{R}^{d}$ by\;
${\mathbb U}^{\{M_p\}}_d$ and the class of all special $\mathfrak{R}$-approximate units on $\mathbb{R}^{d}$ by\;
$\overline{\mathbb U}^{\{M_p\}}_d$.

 \begin{remark}\label{DC}
 By the Denjoy-Carleman theorem, the defined above spaces of ultradifferentiable functions as well as the classes\;
 ${\mathbb U}_d^{\{M_p\}}$ and   $\overline{\mathbb U}_d^{\{M_p\}}$
 of approximate units contain sufficiently many members.
 \end{remark}


\section{Integrability of Roumieu Ultradistributions}


 We formulate below a characterization of integrable Roumieu ultradistributions,
 being an analog of the theorem of Dierolf and  Voigt concerning integrable distributions (see \cite{DiVo})
 and of the theorem of Pilipovi\'{c} concerning ultradistributions of Beurling type (see \cite{Pil91}).
 The proof of the theorem
 is given in \cite{SMK}.

\smallskip
\begin{theorem}\label{thI}
Let $V\in{\cal{D}'}^{\{M_p\}}_d$. The following conditions are equivalent:

\smallskip
\indent
  $(A)$\;\; $V$ is continuous on ${\cal{D}}^{\{M_p\}}_d$
  in the topology induced by
$\dot{\cal{B}}^{\{M_p\}}_d$,
  i.e. there are a sequence $(r_{\!p})\in\mathfrak{R}$ and a constant $C>0$ such that the inequality
\begin{equation}\label{a0}
|\langle V,\varphi\rangle| \leq C \|\varphi\|_{(r_p)}
\end{equation}
holds for all $\varphi \in {\cal{D}}^{\{M_p\}}_d$;

\smallskip
$(B)$\;\; there is a sequence $(r_{\!p})\in\mathfrak{R}$ with the property that for every $\varepsilon>0$
there exists a regular compact set $K\subset \mathbb{R}^d$
such that the inequality
\[
|\langle V,\varphi\rangle| \leq \varepsilon  \|\varphi\|_{(r_p)}
\]
holds for $\varphi \in {\cal{D}}^{\{M_p\}}_d$ with
$\su \varphi \cap K = \emptyset$;

\smallskip
$(C)$\;\; for every $\(\p_n\) \in {\mathbb U} ^{\{M_p\}}_d$
    the sequence $\left(\langle V,\p_n\rangle\right)$ is Cauchy;

\smallskip
$(D)$\;\; for every $\(\p_n\) \in \overline{\mathbb U}^{\{M_p\}}_d$
    the sequence $\left(\langle V,\p_n\rangle\right)$ is Cauchy;

\smallskip
$(E)$\;\; there are a sequence $(r_{\!p})\in\mathfrak{R}$, a constant $C>0$ and a regular compact
$K \subset \mathbb{R}^d$ such that inequality (\ref{a0}) holds
for $\varphi \in {\cal{D}}_d^{\{M_p\}}$ with $\su \varphi
\cap K=\emptyset$.
\end{theorem}


\section{Convolution of Roumieu Ultradistributions}


  S. Pilipovi\'{c} and B. Prangoski made in \cite{PP} a deep study
 of the\ convolution\  of\ Roumieu\
 ultradistributions.
 The study was based on the\ investigation of the $\epsilon$\ tensor\ product
 of the respective spaces of test functions. Let us recall some results proved and observations made in \cite{PP}.

  The authors use the results on the $\varepsilon$ tensor product
  from \cite{Kom3} to prove
  that
      \[
      \dot{{\cal{B}}}^{\{M_p\}}_{d_1}
       \varepsilon\,\,\, \dot{{\cal{B}}}^{\{M_p\}}_{d_2}
       \,\cong\, \dot{{\cal{B}}}^{\{M_p\}}_{d_1}
       \,\widehat{\otimes}_{\varepsilon}\,\, \dot{{\cal{B}}}^{\{M_p\}}_{d_2}
        \]
 in the sense of an isomorphism. They consider, analogously to ideas applied in \cite{ORT-WAG}
 to the convolution of measures,
  the following semi-norms in the space ${\cal{D}}^{\{M_p\}}_{L^{\infty},d}$:
  \[
 \quad q_{g,(r_p)} (\varphi):= \sup_{k \in \mathbb{N}^{d}_0}
 \sup_{x \in\mathbb{R}^d}  \frac{\mid g(x) D^k\varphi (x) \mid }{R_{|k|}M_{k}}, \quad
 \varphi\in{\cal{D}}\,^{\{M_p\}}_{L^{\infty},d}.
 \]
 Denote by $\widetilde{\!\widetilde{\cal{D}}}^{\{M_p\}}_{L^{\infty},d}$
  the l.c.s.  ${\cal{D}}^{\{M_p\}}_{L^{\infty},d}$
 equipped with the topology defined by the family
 $\{q_{g,(r_p)}\!: \ g\in\mathcal{C}_0,\ (r_p)\in\mathfrak{R}\}$ of semi-norms and
 the strong dual of $\widetilde{\!\widetilde{\cal{D}}}^{\{M_p\}}_{L^{\infty},d}$
 by $\left(\widetilde{\!\widetilde{\cal{D}}}^{\{M_p\}}_{L^{\infty},d}\right)_b'$.
 The connection between the strong dual of \, $\widetilde{\!\widetilde{\cal{D}}}^{\{M_p\}}_{L^{\infty},d}$
 and the space of integrable distributions ${{\cal{D}}}'^{\{M_p\}}_{L^{1},d}$ was studied in \cite{PP}.
 Such results have recently improved and it is shown that
\begin{equation}\label{xxx}
\left(\widetilde{\!\widetilde{\cal{D}}}^{\{M_p\}}_{L^{\infty},d}\right)_b' =
{{\cal{D}}}'^{\{M_p\}}_{L^{1},d},
\end{equation}
 as locally convex spaces (cf. \cite{PPV}, Prop\,5.3 and  Prop\,5.4).

These observations allow the authors to give in \cite{PP} the following definitions of
convolvability and
convolution of two Roumieu ultradistributions, analogous to the known
definitions of L. Schwartz for distributions (see \cite{Sch-Sem}):

 \textbf{Definition 4}. \label{Rconv}
 Let $S, T\in {\cal{D}}'^{\{M_{p}\}}_d$.
 If the following condition is satisfied:
  \[
 \begin{array}{llll}
 \!(\es)\qquad\;
 & \!\! \qquad \qquad\qquad
  V_{\varphi} := (S \otimes T)
   \,\varphi^{\triangle} \in {\cal{D}}'^{\{M_p\}}_{L^1,2d}
 & \qquad &\mbox{\ for\ all}\ \varphi\in {\cal{D}}^{\{M_{p}\}}_d,\:
 \end{array}
 \]
 where $\varphi^{\triangle}$ is the function of the  class ${\cal E}^{\{M_p\}}_{2d}$ defined by
 \begin{equation}\label{fitr}
 \varphi^{\triangle}(x, y):=\varphi(x+y),
 \qquad x, y \in\mathbb{R}^d,
 \end{equation}
 then we say that the Roumieu ultradistributions $S, T$ are {\em convolvable} in the sense of $(\es)$.
 Then the {\it convolution} $S \ast T$ of $S$ and $T$ in ${\cal{D}}'^{\{M_{p}\}}_d$
 is defined by
 \begin{equation}\label{P-P}
 \qquad\qquad\quad
 \langle S {\ast} T , \varphi \rangle_d
 := \langle V_{\varphi}, 1 \rangle_{2d},
 \qquad
 \varphi \in {\cal{D}}^{\{M_{p}\}}_d,
 \vspace{1.0mm}
 \end{equation}
where $V_{\varphi}$ is meant, according to (\ref{xxx}), as an element of the space
$\widetilde{\!\widetilde{\cal{D}}}'^{\{M_p\}}_{L^{\infty},2d}$
and  the constant function $1$ is meant
as an element of the space  $\widetilde{\!\widetilde{\cal{D}}}^{\{M_p\}}_{L^{\infty},2d}$.

For each $a>0$ consider the subset
$\triangle_a\!:=\{(x,y)\in \mathbb{R}^{2d}:\ |x+y|\leq a\}$
of $\mathbb{R}^{2d}$ and the following subspace of\,\,
  $\dot{\!\!\widetilde{\cal{B}}}^{\{M_p\}}_{2d}$:
 \begin{equation}\label{zzz}
 \dot{\cal{B}}^{\{M_p\}}_{2d} (\triangle_a)
 := \{ \varphi\in \dot{\!\!\widetilde{\cal{B}}}^{\{M_p\}}_{2d}\!\!:\
            \su \varphi \,\subseteq\, \triangle_a\}.
 \end{equation}
Denote by $\dot{\cal{B}}^{\{M_p\}}_{\triangle,2d}$
the inductive limit of the spaces defined in (\ref{zzz}):
\[
\dot{\cal{B}}^{\{M_p\}}_{\triangle,2d} := \lim_{\begin{subarray}{c}
   \displaystyle\longrightarrow \\
   {\scriptsize a\rightarrow\infty}
   \end{subarray}} \dot{\cal{B}}^{\{M_p\}}_{2d} (\triangle_a).
   \qquad \qquad
\]
The following result on equivalence of convolvability conditions for
Roumieu ultradistributions was proved in \cite{PP}:

\begin{theorem}\label{PP}
 Let $S, T \in {\cal{D}}'^{\{M_{p}\}}_d$.
 The following conditions are equivalent to condition $(\es)$ of convolvability for $S$ and $T$:

 \smallskip
 $(c_0)$ \quad  $S\otimes T \in \dot{\cal{B}}'^{\{M_p\}}_{\triangle,2d}$

\smallskip
  $(c'_1)$ \quad $S(\check{T} \ast \varphi)\in {\widetilde{\cal{D}}}'^{\{M_p\}}_{L^1,d}$
  for all $\varphi\in {\cal{D}}^{\{M_{p}\}}_d$
 and for every compact subset $K$ in $\mathbb{R}^d$, the mapping
  \[
 {\cal{D}}_{K,d}^{\{M_{p}\}}
 \times \dot{\!\widetilde{\cal{B}}}^{\{M_p\}}_d\!
 \ni (\varphi, \chi) \mapsto \left\langle S(\check{T} \ast \varphi), \chi\right\rangle \in\mathbb{C}
 \]
 is a continuous bilinear mapping;

\smallskip
 $(c'_2)$ \quad $(\varphi \ast \check{S}) T \in {\widetilde{\cal{D}}}'^{\{M_p\}}_{L^1,d}$
 for all $\varphi\in {\cal{D}}^{\{M_{p}\}}_d$
 and, for every compact subset $K$ in $\mathbb{R}^d$, the mapping
 \[
 {\cal{D}}_{K,d}^{\{M_{p}\}}
 \times \dot{\!\widetilde{\cal{B}}}^{\{M_p\}}_d
 \!\ni \!(\varphi, \chi)\mapsto\left\langle(\check{S} \ast \varphi) T, \chi\right\rangle\in\mathbb{C}
 \]
 is a continuous bilinear mapping;

\smallskip
$(c_3)$ \quad $(\check{S} \ast \varphi)(T \ast \psi) \in {L^1_d}$
 for all $\varphi, \psi\in {\cal{D}}^{\{M_{p}\}}_d$.
\end{theorem}

 Dimovski, Pilipovi\'{c}, Prangoski  and Vindas
 modified conditions ($c'_1$) and ($c'_2$) and
 proved in \cite{DPPV}, under the assumption that the sequence $(M_p)$ satisfies (M.1), (M.2) and (M.3),
 that they are equiwalent to more transparent versions, as given in the following theorem:

\begin{theorem}\label{DPPV}
 Let $S, T \in {\cal{D}'}^{\{M_{p}\}}_d$.
 The following conditions are equivalent to condition $(\es)$ of convolvability for $S$ and $T$:

 \smallskip
  $(c_1)$ \quad $S(\check{T} \ast \varphi)\in
 {\cal{D}}'^{\{M_p\}}_{L^1,d}$
 for all $\varphi\in {\cal{D}}^{\{M_{p}\}}_d$;

 \smallskip
 $(c_2)$ \quad $(\varphi \ast \check{S}) T \in
{\cal{D}}'^{\{M_p\}}_{L^1,d}$
for all $\varphi\in {\cal{D}}^{\{M_{p}\}}_d$.
 \end{theorem}

In the next section we formulate certain sequential conditions of convolvability
of Roumieu ultradistributions, connected with Theorem\,\ref{thI} on integrability
in ${\cal{D}'}^{\{M_{p}\}}$ from section\,5.


\section{Sequential Definitions of Convolution in ${\cal{D}'}^{\{M_{p}\}}_d$}


The notion of  $\mathfrak{R}$-approximate unit makes us possible to consider several sequential definitions
of the convolution of Roumieu ultradistributions based on
corresponding sequential conditions of convolvability.  The conditions require that respective numerical sequences, corresponding
to a given pair of Roumieu ultradistributions via certain approximate units, are Cauchy sequence (Cauchy s. in short) for all approximate units from a given class.
The first definition of  this kind was given for the convolution of distributions by V. S. Vladimirov in \cite{VLA1}
and its equivalent versions were discussed in \cite{DiVo} and \cite{Kam}. Their counterparts for ultradistributions of Beurling type
 were discussed in \cite{KKP} (see also \cite{CKP}).

 We will prove in Theorem\,\ref{thm2} that all the sequential definitions are equivalent to the definition of the general
convolution of Roumieu ultradistributions in the sense of S. Pilipovi\'{c} and B. Prangoski \cite{PP}.
Our proof of  Theorem\,\ref{thm2} will be based on the integrability result proved in the
previous section.

 \textbf{Definition 5}. \label{convolv}
Let $S, T \in {\cal{D}'}^{\{M_{p}\}}_d$.
 We say that
 $S, T$ \textit{are convolvable
in the sense} of $({\ve})$, $({\p})$,
 $({\p}_1)$, $({\p}_2)$,
if the corresponding condition below holds for every $\varphi \in {\cal{D}}^{\{M_p\}}_d$, respectively:
 \[
 \begin{array}{lll}
 \!({\ve})& \;\;\ \left(\langle S \otimes T,\ \p_n \,\varphi^{\triangle}\rangle_{2d}\right) \;\;
 &\mbox{is\ a\ Cauchy\ s.\ for\ all}\,(\p_n)\in{\mathbb U}^{\{M_p\}}_{2d};\qquad\qquad\quad\quad
 \vspace{1.0mm}
  \\
  \!({\p}) & \;\;\ \left(\langle (\p^1_n S)\otimes (\p^2_n T),\varphi^{\triangle} \rangle_{2d}\right) \;
 &\mbox{is\ a\ Cauchy\ s.\ for\ all}\, (\p^1_n), (\p^2_n)\in{\mathbb U}^{\{M_p\}}_{d}; \quad
 \vspace{1.0mm}
  \\
 \!({\p}_1) & \;\;\ \left(\langle (\p_n S)\otimes T,\varphi^{\triangle} \rangle_{2d}\right) \;\;
 &\mbox{is\ a\ Cauchy\ s.\ for\ all}\,(\p_n)\in{\mathbb U}^{\{M_p\}}_{d};\qquad\qquad\quad\quad
 \vspace{1.0mm}
 \\
 \!({\p}_2) & \;\;\ \left(\langle S\otimes (\p_n T),\varphi^{\triangle} \rangle_{2d}\right) \;\;
 &\mbox{is\ a\ Cauchy\ s.\ for\ all}\,(\p_n)\in{\mathbb U}^{\{M_p\}}_{d}.\qquad\qquad\quad\quad
 \vspace{.2mm} \\
 \end{array}
 \]

 \textbf{Definition 6}. \label{convo}
If $S, T\in {\cal{D}}'^{\{M_{p}\}}_d$
are convolvable in the sense of 
     $({\ve})$, $(\p)$, $({\p}_1)$, $({\p}_2)$,  respectively,
then the \textit{convolution of $S$ and $T$} in ${\cal{D}}'^{\{M_{p}\}}_d$
    in the respective sense is defined by the corresponding
formula below:
 \smallskip
 \begin{eqnarray*}
\displaystyle \langle S\!\stackrel{\ve}{\ast}\!T,\varphi\rangle_d\;\, &\!:=&
        \lim_{n\rightarrow \infty}\langle S \otimes T , \p_n\,\varphi^{\triangle} \rangle_{2d}, \quad
                                        \quad\,\;\ \varphi \in \md_d,\;\;\, (\p_n)\in{\mathbb U}^{\{M_p\}}_{2d};\\
  \bigskip
 \vspace{1.5cm}
  \displaystyle \langle
   S\!\stackrel{\pj}{\ast}\!T, \varphi\rangle_d\;\,
               &\!:=&
        \lim_{n\rightarrow \infty}\langle  (\p^1_n S)\!\otimes\!(\p^2_n T),\varphi^{\triangle}\rangle_{2d}, \;
                                        \;\, \varphi\in \md_d,\;\; (\p^1_n), (\p^2_n)\!\in\!{\mathbb U}^{\{M_p\}}_d; \\
 \displaystyle \langle
      S\!\stackrel{{\pj}_1}{\ast}\!T,\varphi\rangle_d\,
              &\!:=&
        \lim_{n\rightarrow \infty}\langle (\p_nS)\otimes  T, \varphi^{\triangle} \rangle_{2d},
                                        \quad\;\;\,\;\; \varphi \in \md_d,\;\;\,(\p_n)\in{\mathbb U}^{\{M_p\}}_d ; \\
  \bigskip
\vspace{1.5cm}
 \displaystyle
        \langle
     S\!\stackrel{{\pj}_2}{\ast}\!T,\varphi\rangle_d\,
              &\!:=&
        \lim_{n\rightarrow \infty} \langle S\otimes(\p_n T),  \varphi^{\triangle} \rangle_{2d},
                                        \quad\;\;\;\;\,\varphi \in \md_d,\;\;\, (\p_n)\in{\mathbb U}^{\{M_p\}}_{d};
 \end{eqnarray*}
respectively.


 \begin{remark} \label{rem5}
Condition (${\es}$) of convolvability in Definition\,4
guarantees, by the considerations of Pilipovi\'{c} and Prangoski from section\,4 in \cite{PP}, that the convolution $S {\ast} T$
in the sense of (\ref{P-P}) exists in ${\cal{D}}'^{\{M_{p}\}}_d$. It follows from Theorem\,\ref{thm2}, formulated below,
that convolvability conditions $({\ve})$, $(\p)$, $({\p}_1)$ and  $({\p}_2)$  guarantee that the corresponding sequential
convolutions, defined in Definition\,6,
exist in ${\cal{D}}'^{\{M_{p}\}}_d$.

In addition, to the sequential conditions of convolvability in ${\cal{D}}'^{\{M_{p}\}}_d$, given in Definition\,5,
one may consider also the conditions $(\overline{\ve})$, $(\overline{\p})$, $(\overline{\p}_1)$, $(\overline{\p}_2)$
being the modifications of the above ones consisting in replacing the classes ${\mathbb U}^{\{M_p\}}_{2d}$ and
${\mathbb U}^{\{M_p\}}_{d}$ of $\mathfrak{R}$-approximate units by the classes $\overline{\mathbb U}^{\{M_p\}}_{2d}$ and
$\overline{\mathbb U}^{\{M_p\}}_{d}$ of special approximate units, respectively. The modified conditions lead to additional sequential
definitions of the convolution in ${\cal{D}}'^{\{M_{p}\}}_d$ which are counterparts of the known sequential definitions of the
convolution of distributions (see \cite{VLA1}, \cite{DiVo} and \cite{Kam}). It follows from Theorem\,\ref{thI} that they
are equivalent to all the conditions listed in Definition\,6.
\end{remark}

\begin{remark}\label{rem6}
The convolution of Roumieu ultradistributions in ${\cal{D}}'^{\{M_{p}\}}_d$ investigated in \cite{PP} and \cite{DPPV}
and its sequential versions discussed in this paper is a general notion. It embraces various particular cases, e.g.
expressed in terms of supports of given Roumieu ultradistributions. The discussion of such sufficient conditions
will be given in forthcoming papers.

\smallskip
Each of the
sequential definitions
of the convolution of $S$ and $T$,
under the corresponding
conditions, do not depend on the choice of an approximate unit from a given class.
\end{remark}

 \begin{theorem}\label{thm2}
 Let $S, T \in \mdd_d$.
  All the conditions formulated in Definition\,\ref{convolv} are equivalent to
 the condition $({\es})$ of convolvability of ultradistributions.
 If one of these conditions is satisfied,
 then all corresponding convolutions of Roumieu ultradistributions $S$ and $T$
 exist in $\mdd_d$ and all the convolutions are equal:
 \[
 S\!\stackrel{{\ve}}{\ast}\!T = S\!\stackrel{\pj}{\ast}\!T =  S\!\stackrel{{\pj}_1}{\ast}\!T
 =  S\!\stackrel{{\pj}_2}{\ast}\!T = S {\ast} T.
 \]
 \end{theorem}

\begin{proof}
 We will prove the equivalence of convolvability conditions given in Definitions \ref{convolv} and  \ref{Rconv} and in Theorem \ref{DPPV}
 according to the following scheme of implications:

 \begin{displaymath}
 \hspace{2.2cm}\begin{array}{ccccccccccc}
  & & ({\p}_1) & \longrightarrow & (\ce_1) & & & & & &   \\
  &\nearrow & & & & \searrow & & & & &     \\
 ({\p}) & & & & & &  (\es) & \longrightarrow &({\ve}) & \longrightarrow & ({\p})   \\
  & \searrow & & & & \nearrow & &  & & &  \\
  & & ({\p}_2) & \longrightarrow & (\ce_2) & & & & & &
  \end{array}
 \end{displaymath}

 Assume condition $({\p})$ for a fixed $\varphi\in \md_d$.
 This means that the following limit exists for all
 $(\p^1_n), (\p^2_n)\in {\mathbb U}^{\{M_p\}}_d$:
 \[
 \lim_{n\rightarrow \infty}\langle (\p^1_n S)\!\otimes\!(\p^2_n T),\varphi^{\triangle}\rangle_{2d} = \alpha
 \]
 for a certain $\alpha\in\mathbb{C}$ which does not depend on the sequences
 $(\p^1_n)$ and $(\p^2_n)$.
 Hence the double limit
 \[
 \lim_{i, j\to\infty} \langle(\p^1_i S)\!\otimes\!(\p^2_j T),\varphi^{\triangle}\rangle_{2d}= \alpha
 \]
 also exists for all $(\p^1_i), (\p^2_j)\in {\mathbb U}^{\{M_p\}}_d$.
 Indeed, if the last equality is not true, then there exist increasing sequences of positive integers
 $(i_n)$ and $(j_n)$ such that
 \[
 \left|\langle(\p^1_{i_n} S)\!\otimes\!(\p^2_{j_n} T),\varphi^{\triangle}\rangle - \alpha\right| > \varepsilon,
 \]
 for all $n\in\mathbb{N}$, but this conducts to contradiction with condition  $({\p})$,
 because  $(\p^1_{i_n}), (\p^2_{j_n})\in {\mathbb U}^{\{M_p\}}_d$.
 Thus under the existence of the iterated limits and definition of approximate unit we obtain
 \[
 \alpha = \lim_{n\rightarrow \infty}\langle (\p^1_nS)\otimes  T, \varphi^{\triangle} \rangle_{2d},
 \]
 which ensure  condition $({\p}_1)$.

 Next for a fixed $\varphi\in \md_d$ and arbitrary  $(\p_n) \in {\mathbb U}^{\{M_p\}}_d$
 we may write
 \[
 \langle (\p_nS)\otimes  T, \varphi^{\triangle} \rangle_{2d} =
 \langle(\p_nS)\ast T, \varphi\rangle_{d} = \langle \p_nS, \varphi\ast\check{T}\rangle_{d}
 = \langle S (\check{T}\ast\varphi), \p_n \rangle_d, \quad  n\in\mathbb{N}
 \]
 for any function $\p_{n}\in \md_d$.
  Then according to Theorem\,\ref{thI} and condition $({\p}_1)$ we have the equality
 \[
 \lim_{n\rightarrow \infty} \langle (\p_nS)\otimes  T, \varphi^{\triangle} \rangle_{2d}
 = \langle S (\check{T}\ast\varphi), 1 \rangle
 \]
 because the definition of family ${\mathbb U}^{\{M_p\}}_d$. Moreover, this implies integrability of ultradistribution
 $S (\check{T}\ast\varphi)$ for all $\varphi\in \md_d$.

 The proof of implications $({\p})\Rightarrow({\p}_2)\Rightarrow(c_2)$ is analogous. The equivalence of conditions
 ($c_1$), ($c_2$) and (${\es}$) is given by Theorem \ref{DPPV}. From Theorem \ref{thI} with
 $V := (S \otimes T)\,\varphi^{\triangle}$ and  $(\p_n)\in{\mathbb U}^{\{M_p\}}_{2d}$ we obtain equivalence
 of conditions (${\es}$) and (${\ve}$).

 It remains to note that the implication  $({\ve})\Rightarrow({\p})$ is obvious, since for every $\varphi\in \md_d$
 and arbitrary  $(\p^1_n), (\p^2_n) \in {\mathbb U}^{\{M_p\}}_d$ the functions $(\p^1_{n}\otimes \p^2_{n})\varphi^{\triangle}$
 belong to $\md_{2d}$ for $n\in \mathbb{N}$ and
 \[
 \langle (\p^1_n S)\!\otimes\!(\p^2_n T),\varphi^{\triangle}\rangle_{2d} =
 \langle S\!\otimes\!T,\varphi^{\triangle}(\p^1_{n}\otimes \p^2_{n})\rangle_{2d}
 \]
 according to the notion of tensor products in spaces
 $\md_d$ and $\mdd_d$ and the respective topological
 isomorphisms described in Theorems 2.1 and 2.3 in \cite{Kom2}.
 Condition (${\ve}$) guaranties that the sequence
 $\left(\langle (\p^1_n S)\!\otimes\!(\p^2_n T),\varphi^{\triangle}\rangle_{2d}\right)$ is Cauchy
 because $(\p^1_{n}\otimes \p^2_{n})\in {\mathbb U}^{\{M_p\}}_{2d}$.

\end{proof}

\begin{remark} \label{remstr}
In addition, one may consider the sequential conditions
$(\overline{\ve})$, $(\overline{\p})$, $(\overline{\p}_1)$, $(\overline{\p}_2)$
of convolvability of $S, T\in\mdd_d$  and the corresponding convolutions
$S\stackrel{\overline{\ve}}{\ast}T$, $S\!\stackrel{\overline{\p}}{\ast}\!T$,
$S\!\stackrel{\overline{\p} _1}{\ast}\!T$, $S\!\stackrel{\overline{\p}_2}{\ast}\!T$
of $S$ and $T$, replacing the classes ${\mathbb U}^{\{M_p\}}_{2d}$ and
${\mathbb U}^{\{M_p\}}_{d}$ of $\mathfrak{R}$-approximate units by the classes $\overline{\mathbb U}^{\{M_p\}}_{2d}$ and
$\overline{\mathbb U}^{\{M_p\}}_{d}$ of special $\mathfrak{R}$-approximate units.
The modified conditions of convolvability and the corresponding sequential definitions of the convolution
of $S$ and $T$ are equivalent to those considered above.

The equivalence of conditions $({\es})$, $(\overline{\ve})$, $(\overline{\p})$, $(\overline{\p}_1)$ and $(\overline{\p}_2)$
of convolvability follows in the same manner as the equivalence of conditions $({\es})$,
$({\ve})$, $({\p})$, $({\p}_1)$ and $({\p}_2)$ proved in Theorem\,\ref{thm2} above. The equality
of the corresponding convolutions is an easy consequence of this result.
\end{remark}

\begin{remark}\label{comut}
It is easy to see that the convolution  of Roumieu ultradistributions
is commutative, i.e. $S{\ast}T = T{\ast}S$ for
$S, T \in \mdd_d$.
\end{remark}


\section{Ultradifferential Property of the Convolution}


Let us consider an ultradifferential operator $P(D)$ defined by Komatsu in \cite{Kom1} as follows:

 \textbf{Definition 7}. \label{ultD}
An operator of the form
 \begin{equation}\label{PD}
P(D) = \sum_{k\in\mathbb{N}_0^d} c_kD^k, \quad c_k\in\mathbb{C}
 \end{equation}
is called an \textit{ultradifferential operator of class} $\{M_p\}$ if for every $L>0$ there is
a constant $C_L$ such that
\begin{equation}\label{ckPD}
|c_k| \leq C_L\frac{L^k}{M_k}, \qquad k\in\mathbb{N}_0^d.
 \end{equation}
The estimation in (\ref{ckPD}) means that for every $L>0$ we have $\sup_k\left(L^{-k}{M_k|c_k|}\right)< \infty$.
Then according to Lemma\,\ref{Kom} Part\,($I\!I$) there is a sequence $(u_p)\in\mathfrak{R}$ such that
$\sup_k\left(U_{|k|}M_k|c_k|\right)< \infty$ where $U_k=\prod_{p\leq k}u_p$. In other words an ultradifferential operator
of the form (\ref{PD}) is of class $\{M_p\}$ if there are $C>0$ and $(u_p)\in\mathfrak{R}$ such that
\begin{equation}\label{ckPDu}
|c_k| \leq \frac{C}{U_{|k|}M_k}, \qquad k\in\mathbb{N}_0^d.
 \end{equation}

Due to Komatsu \cite{Kom1}, if $(M_p)$ satisfies condition (M.2), then $P(D)$ defines the respective continuous mappings
$\md_d\rightarrow\md_d$ and $\mdd_d\rightarrow\mdd_d$. Moreover the series
$ \displaystyle P(D)S = \sum_{k\in\mathbb{N}_0^d} c_kD^kS$ converges absolutely in $\mdd_d$ for every  $S\in \mdd_d$ .

In Theorem\,\ref{thm3} below we prove an important and non-trivial property of the convolution of
Roumieu ultradistributions. In the proof
we will need the following very useful result from \cite{PP}:

\begin{lemma}\label{PiPr}
For every sequence $(s_{\!p})\in\mathfrak{R}$ there exists a sequence $(r_{\!p})\in\mathfrak{R}$ such that $r_{\!p}\leq s_{\!p}$
for $p\in\mathbb{N}$ and
\begin{equation}\label{inPP}
R_{p+q} \leq 2^{p+q}R_pR_q  \quad \mbox{for \ all} \quad p, q\in\mathbb{N}_0.
 \end{equation}
 \end{lemma}

 \begin{theorem}\label{thm3}
 Let $S, T \in \mdd_d$ be convolvable and let $P(D)$ be an ultradifferential operator of class $\{M_p\}$.
 Then $P(D)S$ and $T$ as well as $S$ and $P(D)T$ are convolvable and, moreover,
 \begin{equation}\label{PDST}
 P(D)(S\ast T) = (P(D)S)\ast T = S\ast(P(D)T).
 \end{equation}
 \end{theorem}

\begin{proof}
 Assume that $S, T \in \mdd_d$ are convolvable.
 By the definitions of $P(D)$ and $S{\ast}T$ and by Theorem\,\ref{thI}, we have
\begin{eqnarray}
 \langle P(D)(S\ast T), \varphi\rangle &=& \langle S\ast T, P(-D)\varphi\rangle =
 \langle(S\otimes T)(P(-D_x)\varphi)^{\triangle}, 1_{2d}\rangle \nonumber \\
 &=& \lim_{n\rightarrow\infty}\langle S\otimes T, \p_n\left[P(-D_x)\varphi^{\triangle}\right]\rangle \label{PDlim1}
\end{eqnarray}
 for all $(\p_n)\in \overline{\mathbb U}^{\{M_p\}}_{2d}$ and  $\varphi\in \md_d$.

 We are going to prove only that the ultradistributions $P(D)S$ and $T$ are convolvable in $\mdd_d$
 and the first equality of (\ref{PDST}) holds, because the remaining part of the assertion follows then
 directly from Remark\,\ref{comut}.
  To prove the convolvability of $P(D)S$ and $T$ we have to show that the sequence
 $\left(\langle P(D)S\otimes T, \p_n\varphi^{\triangle}\rangle\right)_{n\in\mathbb{N}}$ is convergent.

 We have
 \begin{eqnarray}
 \langle P(D)S\otimes T, \p_n\varphi^{\triangle}\rangle &=& \hspace{-.2cm}
 \sum_{\alpha\in\mathbb{N}_0^d}\!c_{\alpha}\langle D_x^{\alpha}S\otimes T, \p_n\varphi^{\triangle}\rangle
 = \hspace{-.2cm} \sum_{\alpha\in\mathbb{N}_0^d}\!c_{\alpha}\langle S\otimes T, - D_x^{\alpha}\left(\p_n\varphi^{\triangle}\right)\rangle
   \nonumber \\ &=&
   \langle  S\otimes T, P(-D_x)\left(\p_n\varphi^{\triangle}\right)\rangle \label{PDlim2}
 \end{eqnarray}
 for all $n\in\mathbb{N}$, by (\ref{PD})
  and the absolute convergence of the respective series.

  Comparing the last terms in (\ref{PDlim1}) and (\ref{PDlim2}), we see that the declared assertion
 will be proved if we show the equalities
 \begin{equation}\label{PDni}
 P(-D_x)\left(\p_n\varphi^{\triangle}\right) = \p_nP(-D_x)\varphi^{\triangle} + \nu_n
 \end{equation}
 on $\mathbb{R}^{2d}$ for all $n\in\mathbb{N}$, where
 $\nu_n$ are certain functions in $\md_d$ such that
\begin{equation}\label{STni}
\lim_{n\to \infty}\langle S\otimes T, \nu_n\rangle = 0.
\end{equation}

 Applying (\ref{PD}) and  Leibniz' rule and then changing the order of summation, we get,
 for all $n\in\mathbb{N}$, the equalities
 \begin{eqnarray*}
 P(-D_x)\left(\p_n\varphi^{\triangle}\right) &=&\hspace{-.2cm}
\sum_{\alpha\in\mathbb{N}^d_0}(-1)^{|\alpha|}c_{\alpha}\sum_{i\leq\alpha}{\alpha \choose i}
(D_x^i\p_n)\,(D^{\alpha-i}\varphi)^{\triangle}  \\
&=&\hspace{-.2cm}\sum_{i\in\mathbb{N}^d_0}\hspace{-.1cm}D_x^i\p_n\hspace{-.1cm}\sum_{\beta\in\mathbb{N}^d_0}\hspace{-.1cm}
(-1)^{|\beta+i|}{\beta+i \choose i}c_{\beta+i}(D^{\beta}\varphi)^{\triangle}  \\
&=& \p_nP(-D_x)\varphi^{\triangle} + \nu_n
 \end{eqnarray*}
on $\mathbb{R}^{2d}$ with the functions $\nu_n$ defined for $x, y\in\mathbb{R}^d$ by
\begin{equation}\label{nin}
\nu_n(x, y) := \sum_{i\in\mathcal{N}}\hspace{-.1cm}D_x^i\p_n(x, y)\hspace{-.1cm}\sum_{\beta\in\mathbb{N}^d_0}\hspace{-.1cm}
(-1)^{|\beta+i|}{\beta+i \choose i}c_{\beta+i}D_x^{\beta}\varphi(x + y),
\end{equation}
 where ${\mathcal{N}}:=\mathbb{N}^d_0\setminus\{(0, \ldots, 0)\}$.
 This means equations (\ref{PDni}) hold for $\nu_n$ defined in (\ref{nin}) and for all $n\in\mathbb{N}$.
 Clearly, $(\nu_n)$ depends on the initial sequence $(\p_n)$.

 It suffices to show (\ref{STni}).
Choose $\theta\in \md_d$ such that $\theta(x)=1$ for $x\in\su \varphi$. By (\ref{nin}), we have
\[
\langle S\otimes T, \nu_n\rangle = \langle (S\otimes T)\theta^{\triangle}, \nu_n\rangle, \quad n\in\mathbb{N}.
\]
Due to assumption that  $S$ and $T$ are convolvable in $\mdd_d$, the sequence $\left(\langle (S\otimes T)\theta^{\triangle}, \overline{\p}_n\rangle\right)_{n\in\mathbb{N}}$ is  convergent
for every $(\overline{\p}_n)\in\overline{\mathbb U}^{\{M_p\}}_{2d}$.
To prove (\ref{STni}) it is enough to show that also
$(\overline{\p}_n+\nu_n)\in\overline{\mathbb U}^{\{M_p\}}_{2d}$.
 Since $(\p_n), (\overline{\p}_n)\in\overline{\mathbb U}^{\{M_p\}}_{2d}$, for each
 compact set $K$ in $\mathbb{R}^{2d}$ there exists an $n_0\in\mathbb{N}$ such that
 $D^i\p_n(x, y)=0$ and $\overline{\p}_n(x, y)=1$ for $(x, y)\in K, i\in\mathcal{N}$ and $n>n_0$.
 Consequently, in view of (\ref{nin}),
 \[
 \overline{\p}_n(x, y)+\nu_n(x, y)=1 \qquad \mbox{for} \; (x, y)\in K, \;  \; n>n_0.
 \]
 Therefore it remains to prove, for every $(t_{\!p})\in\mathfrak{R}$, that
 \begin{equation}\label{Crp}
 \sup_{n\in\mathbb{N}}\|\nu_n\|_{(t_{\!p})}<\infty,
 \end{equation}
  since  $(\overline{\p}_n)\in\overline{\mathbb U}^{\{M_p\}}_{2d}$ and (\ref{Crp}) implies
 $ \sup_{n\in\mathbb{N}}\|\overline{\p}_n + \nu_n\|_{(t_{\!p})}<\infty$.

 Fix an arbitrary $(t_{\!p})\in\mathfrak{R}$.
The coefficients of the ultradifferential operator $P(D)$ satisfy (\ref{ckPDu})
for some $(u_{\!p})\in\mathfrak{R}$. Putting $s_{\!k}:=\min\{t_{\!k}, u_{\!k}\}$ for $k\in\mathbb{N}$, we have $(s_{\!p})\in\mathfrak{R}$.
By Lemma\,\ref{PiPr}, there exists a sequence $(r_{\!p})\in\mathfrak{R}$ such that $r_{\!k}\leq s_{\!k}$
and inequality (\ref{inPP}) holds. In addition we assume,  according to Remark\,\ref{rem0rn}, that
  \begin{equation}\label{rp2H}
 r_{\!p}>{16H^2} \quad \mbox{for} \quad p\in\mathbb{N},
  \end{equation}
 where $H>\frac{1}{4}$ is a constant from condition (M.2).

 Let $\alpha\in\mathbb{N}^d_0$ be arbitrarily fixed. For fixed $n\in\mathbb{N}$ and $(r_{\!p})\in\mathfrak{R}$ chosen
 above, according to the representation (\ref{nin}), we have
 \begin{eqnarray}
 \frac{\left\|D_x^{\alpha}\nu_n\right\|_{\infty}}{R_{|\alpha|}M_{\alpha}}&\leq&
  \sum_{i\in\mathcal{N}}
     \sum_{j\leq\alpha}{\alpha \choose j}\frac{\|D_x^{\alpha+i-j}\p_n\|_{\infty}}{R_{|\alpha+i-j|}M_{\alpha+i-j}}
     \sum_{\beta\in\mathbb{N}^d_0}{\beta+i \choose i}|c_{\beta+i}| \nonumber \\
   &&\cdot\frac{\|D_x^{\beta+j}\varphi^{\triangle}\|_{\infty}}{R_{|\beta+j|}M_{\beta+j}}\cdot\!
     \frac{R_{|\alpha+i-j|}R_{|\beta+j|}}{R_{|\alpha|}}\cdot\!\frac{M_{\alpha+i-j}M_{\beta+j}}{M_{\alpha}}. \label{noni}
 \end{eqnarray}

 Applying  properties (\ref{Rpq}) and (\ref{inPP}) (twice) of the sequence $(r_{\!p})$ and
  property (\ref{Mpq}) and condition (M.2) (twice) of the sequence $(M_p)$, we get
  \begin{equation}\label{R2R}
 \frac{R_{|\alpha+i-j|}R_{|\beta+j|}}{R_{|\alpha|}R_{|\beta|}R_{|i|}}\leq
 \frac{2^{|\beta+j|}R_{|\alpha+i|}}{R_{|\alpha|}R_{|i|}}\leq 2^{|\alpha+i|}2^{|\beta+j|}
  \end{equation}
and \begin{equation}\label{MHM}
 \frac{M_{\alpha+i-j}M_{\beta+j}}{M_{\alpha}M_{\beta}M_i}\leq \frac{AH^{|\beta|+|j|}M_{\alpha+i}}{M_{\alpha}M_i}
 \leq A^2H^{|\alpha+i|}H^{|\beta+j|}.
 \end{equation}
 Moreover, we have
  \begin{equation}\label{cBei}
 |c_{\beta+i}| \leq \frac{C}{U_{|\beta|}U_{|i|}\,M_{\beta}M_i}
 \qquad \mbox{and} \qquad
 \frac{R_{|\beta|}R_{|i|}}{U_{|\beta|}U_{|i|}}\leq 1,
  \end{equation}
 by (\ref{ckPDu}), properties (\ref{Mpq}) of $(M_p)$ and (\ref{Rpq}) of $(u_{\!p})\in\mathfrak{R}$
 and because $r_{\!k}\leq u_{\!k}$ for $k\in\mathbb{N}$.
 The inequalities in (\ref{R2R}), (\ref{MHM}) and (\ref{cBei}) hold for arbitrary $\alpha, \beta, i, j\in\mathbb{N}_0^d$
 such that $j\leq\alpha$ and will be used later together with the following known estimate:
  \begin{equation}\label{Bpoi}
{\beta+i \choose i}\leq 2^{|\beta+i|},
   \end{equation}
 for a certain $B>0$ and all $\beta, i\in\mathbb{N}_0^d$.

 It follows from (\ref{noni}), due to (\ref{R2R})-(\ref{Bpoi}), that
 \begin{eqnarray}
 \frac{\left\|D_x^{\alpha}\nu_n\right\|_{\infty}}{R_{|\alpha|}M_{\alpha}}&\leq&
 A^2C\sum_{i\in\mathcal{N}}
     \sum_{j\leq\alpha}{\alpha \choose j}(2H)^{|\alpha+i|}\frac{\|D_x^{\alpha+i-j}\p_n\|_{\infty}}{R_{|\alpha+i-j|}M_{\alpha+i-j}}
    \nonumber \\
 &&\cdot\sum_{\beta\in\mathbb{N}^d_0}2^{|\beta+i|}(2H)^{|\beta+j|}
   \frac{\|D_x^{\beta+j}\varphi^{\triangle}\|_{\infty}}{R_{|\beta+j|}M_{\beta+j}}.  \label{nonial}
   \end{eqnarray}

 According to assumption (\ref{rp2H}), consider
 the sequences $(\overline{r}_{\!p})$ and $(\overline{\overline{r}}_{\!p})$ of the class $\mathfrak{R}$ defined by
 $\overline{r}_{\!p}:=r_{\!p}/{8H}$ and $\overline{\overline{r}}_{\!p}:=r_{\!p}/{16H^2}$, respectively, for $p\in\mathbb{N}$.
 Clearly,
 \[
 2^{2|\alpha+i-j|}(2H)^{|\alpha+i-j|}\frac{\|D_x^{\alpha+i-j}\p_n\|_{\infty}}{R_{|\alpha+i-j|}M_{\alpha+i-j}}
 \leq \|\p_n\|_{(\overline{r}_{\!p})}
 \]
 and
 \[
 2^{2|\beta+j|}(2H)^{2|\beta+j|}\frac{\|D_x^{\beta+j}\varphi^{\triangle}\|_{\infty}}{R_{|\beta+j|}M_{\beta+j}}
 \leq \|\varphi\|_{(\overline{\overline{r}}_{\!p})}
 \]
 for all $\alpha, \beta, i, j\in\mathbb{N}_0^d$, $j\leq\alpha$. We deduce from (\ref{nonial}) and the
 above estimates that
 \[
 \frac{\left\|D_x^{\alpha}\nu_n\right\|_{\infty}}{R_{|\alpha|}M_{\alpha}} \leq
 \frac{A^2C}{2^{\alpha}}\|\p_n\|_{(\overline{r}_{\!p})}\|\varphi\|_{(\overline{\overline{r}}_{\!p})}
 \cdot
 \sum_{i\in\mathcal{N}}\left(\frac{1}{2}\right)^i
 \sum_{\beta\in\mathbb{N}^d_0}\left(\frac{1}{4H}\right)^{\beta}<\infty.
 \]
for arbitrary $\alpha \in\mathbb{N}_0^d$ and $n\in\mathbb{N}$. Hence
 \[
\sup_{n\in\mathbb{N}}\|\nu_n\|_{(r_{\!p})} = \sup_{n\in\mathbb{N}}
\sup_{\alpha\in\mathbb{N}_0^d}\frac{\left\|D_x^{\alpha}\nu_n\right\|_{\infty}}{R_{|\alpha|}M_{\alpha}}<\infty
\]
and, since  $r_{\!p}\leq t_{\!p}$ for $p\in\mathbb{N}$,
\[\sup_{n\in\mathbb{N}}\|\nu_n\|_{(t_{\!p})} \leq
\sup_{n\in\mathbb{N}}\|\nu_n\|_{(r_{\!p})}<\infty,
\]
i. e. (\ref{Crp}) is proved, as required. The assertion of Theorem\,\ref{thm3} is proved.
\end{proof}

Theorem\,\ref{thm3} has also been
shown in the quasianalytic case in the forthcoming article
\cite{PPV}. The proof there is via a completely different method (cf. \cite{PPV}, Cor.\,5.10).

\section*{Acknowledgements}

This work was partly supported by the Center for
Innovation and Transfer of Natural Sciences and Engineering
Knowledge of University of Rzesz\'ow.

\end{document}